\documentclass[a4paper,10pt]{article}
\usepackage[utf8]{inputenc}
\usepackage{geometry}
\usepackage{color}
\usepackage{graphicx}
\usepackage{amssymb}
\usepackage{epstopdf}
\usepackage{mathrsfs}
\usepackage{amsmath}
\usepackage{amsthm}
\usepackage{enumerate}
\usepackage[dvipsnames]{xcolor}
\usepackage{hyperref}
\newcommand\R{\mathbb{R}}
\renewcommand{\Im}{\operatorname{Im}}
\renewcommand{\Re}{\operatorname{Re}}

\newcommand\N{\mathbb{N}}
\newcommand\C{\mathbb{C}}

\newcommand\E{\mathbb{E}}

\newcommand\kQ{\mathfrak{Q}}

\newcommand{\cQ}{\mathcal{Q}}

\newcommand{\cP}{\mathcal{P}}
\newcommand{\cL}{\mathcal{L}}

\newcommand{\cS}{\mathcal{S}}

\def\Xint#1{\mathchoice
   {\XXint\displaystyle\textstyle{#1}}%
   {\XXint\textstyle\scriptstyle{#1}}%
   {\XXint\scriptstyle\scriptscriptstyle{#1}}%
   {\XXint\scriptscriptstyle\scriptscriptstyle{#1}}%
   \!\int}
\def\XXint#1#2#3{{\setbox0=\hbox{$#1{#2#3}{\int}$}
     \vcenter{\hbox{$#2#3$}}\kern-.5\wd0}}

\def\dashint{\Xint-}

\usepackage[most]{tcolorbox}

\begin{document}

\newcommand{\dd}{\mathrm{d}}
\newcommand{\ee}{\mathrm{e}}
\newcommand{\ii}{\mathrm{i}}
\newtheorem{theoreme}{Theorem}
\newtheorem{definition}{Definition}
\newtheorem{lemme}{Lemma}
\newtheorem{rem}{Remark}
\newtheorem{exemple}{Example}
\newtheorem{proposition}{Proposition}
\newtheorem{corollaire}{Corollary}
\newtheorem{hyp}{Hypothesis}
\newtheorem*{theo}{Theorem}
\newtheorem*{prop}{Proposition}
\newtheorem*{Conj}{Conjecture}
\newtheorem*{main*}{Main results}
\newtheorem*{LemStar}{Lemma}

\newcommand{\mar}[1]{{\marginpar{\sffamily{\scriptsize
        #1}}}}
\newcommand{\mi}[1]{{\mar{MI:#1}}}

\title{A trace formula for scattering resonances of unbalanced quantum graphs}
\author{Maxime Ingremeau\footnote{Université Côte d'Azur, Laboratoire J.A. Dieudonné}}
\date{}

\maketitle

\begin{abstract}
Given an unbalanced open quantum graph, we derive a formula relating sums over its scattering resonances with integrals outside a strip. We deduce lower bounds on the number of resonances (in bounded regions of the complex plane) that are independent of the size of the graph. We also deduce partial results indicating that Benjamini-Schramm convergence of open quantum graphs should imply convergence of the empirical spectral measures.
\end{abstract}

\section{Introduction}

Quantum graphs are singular one-dimensional objects (i.e., graphs in which edges are seen intervals), on which waves can propagate, following some transmission conditions at the vertices; from the physical point of view, the Kirchhoff conditions, which we consider all along the paper, are the most natural.

The study of spectral properties of quantum graphs has had a growing popularity in the last decades. The reason for this popularity is threefold: the spectrum of quantum graphs is easy to study numerically (and, by certain aspects, theoretically), due to the fact that the eigenfunctions of quantum graphs are complex exponentials on each edge; yet, they can model a wide variety of situations of physical relevance; finally, studying spectral properties in the simplified setting of quantum graphs can give an insight on more complicated situations, for instance involving Schrödinger operators in $\R^d$. This last point became particularly clear since the work of Kottos and Smilansky \cite{KS97, KotSmi}, where they show that quantum graphs share spectral properties with quantum chaotic systems.
We refer the reader to the monograph \cite{BK} for an overview of the recent developments concerning quantum graphs and their applications.

In this paper, we will be interested in \emph{open} quantum graphs, containing some semi-infinite edges. In such open systems where waves can escape towards infinity, the natural spectral objects to study are the \emph{scattering resonances}. They are complex numbers, associated with resonant states, which are ``generalized eigenfunctions" growing exponentially at infinity. When such idealized states are propagated by the wave equation, the real part of the resonance dictates the speed of oscillation, while the imaginary part gives the rate at which the wave escapes towards infinity. The scattering resonances can be seen as the eigenvalues of a non-selfadjoint operator, and are thus often more delicate to understand than the eigenfunctions of genuine self-adjoint Schrödinger operators.
We refer the reader to the book \cite{DZ} for an introduction to the theory of scattering resonances for Schrödinger operators, and for an account of its recent developments.

\subsubsection*{Previous results on the resonance counting of a quantum graph}

If $\cQ$ is a quantum graph, we will denote by $\cL_\cQ$ its total length, and by $\mathrm{Res}(\cQ)$ the set of its resonances, whose definition will be recalled in section \ref{sec:Res}. If $\Omega\subset \C$, we denote by $\mathcal{N}_\cQ(\Omega)$ the number of resonances of $\cQ$ inside $\Omega$, counted with multiplicities

\paragraph{Weyl and non-Weyl graphs}

A graph is called \emph{unbalanced} if, for any edge, the number of finite and of infinite edges attached to this edge are different. It turns out that this condition plays an essential role when studying the resonances of a quantum graphs (with Kirchhoff boundary conditions), as was first shown in  \cite{NonWeyl}:
\begin{tcolorbox}
\begin{theoreme}[Davies-Pushnitski (2012)]\label{Th:DavPush}
Let $\cQ$ be a quantum graph. We have
$$ \mathcal{N}_\cQ \left(D(0,R)\right) = \frac{2}{\pi} \mathcal{L}_\cQ W_\cQ + O_{R\to \infty}(1)$$
for some $W_\cQ \in (0,1]$. Furthermore, if the graph is unbalanced, we have $W_\cQ = 1.$
\end{theoreme}
\end{tcolorbox}

When the graph is unbalanced, this result is similar to the classical Weyl's law for self-adjoint Scrödinger operators. Analogues of Weyl's law are known for resonances of one dimensional Schrödinger operators (see \cite{ZwoBound} and the references therein), but in higher dimensions, only upper bounds are known (\cite{ZwoMultiD}). 

When the graph is balanced, it has fewer resonances, and is hence called non-Weyl. Other situations leading to non-Weyl asymptotics include the presence of magnetic fields, more general vertex coupling conditions... We refer the reader to \cite{Lip} for an account these recent developments.

A convenient way of studying the resonances of a quantum graph is through the secular equation, which was introduced by Kottos and Smilansky.
Namely, there exists a holomorphic family of matrices $U_\cQ(z)$ such that, for all $z\neq 0$,
\begin{equation}\label{eq:CritRes}
z \text{ is a resonance } \Longleftrightarrow \mathrm{\det} (\mathrm{Id} - U_\cQ(z)) = 0,
\end{equation}
and the multiplicities coincide\footnote{Actually, though the proof of (\ref{eq:CritRes}) has been given several times (see \cite{KotSmi, Equivalence, Ing}, as far as the author knows, \cite{NonWeyl} is the only place where the equivalence of multiplicities is proven.}. When $z=0$, multiplicities need not coincide (see \cite{Fulling} for the case of closed quantum graphs),  but in the results below,  the multiplicity of a resonance $z_0$ is to be understood as the multiplicity of $z\mapsto \mathrm{\det} (\mathrm{Id} - U_\cQ(z))$, even when $z=0$.

We refer the reader to section \ref{sec:Secular} for the definition of the matrix $U_\cQ$. A key ingredient in the proof of Theorem \ref{Th:DavPush} is that the function $\mathrm{\det} (\mathrm{Id} - U_\cQ(z))$ is a linear combination of complex exponentials. One can thus us the classical results given in \cite{Langer} about the zeroes of such functions, making Theorem \ref{Th:DavPush} easier to prove than its analogue for general Schrödinger operators.

\paragraph{Imaginary parts of resonances}

A consequence of equation (\ref{eq:CritRes}) and of the results of \cite{Langer} is that the resonances do all lie in vertical strip $R+ i[-K, 0]$ for some $K>0$. This was stated in \cite{NonWeyl}, without any expression for $K$. In \cite{Ing}, an expression was given in the case of \emph{unbalanced} quantum graphs, involving the minimal length, and the maximal internal and external degrees of the graph.

Namely, let us denote by $\mathfrak{Q}_{D, n_0, L_{min}, L_{max}}$ the set of open quantum graphs whose finite edges have lengths between $L_{min}$ and $L_{max}$, and such that every vertex has at most $D$ finite edges and $n_0$ infinite edges attached to it. We denote by $\mathfrak{Q}'_{D, n_0, L_{min}, L_{max}}$ the set of such quantum graphs that are unbalanced.
 
If $\cQ$ belongs to $\mathfrak{Q}'_{D, n_0, L_{min}, L_{max}}$, then we have
\begin{equation}\label{eq:ResonancesInStrip}
\mathrm{Res}(\cQ)\subset \R + i \left[Y(D, n_0, L_{min}), 0\right],
\end{equation}

with
\begin{equation}\label{eq:DefY}
 Y(D, n_0, L_{min}) := - \frac{\ln(D+n_0)}{L_{min}}.
 \end{equation}
 
 The proof of this result is elementary, and will be recalled in section \ref{seq:ProofStrip}.
 
As far as the author knows, no explicit bound for the imaginary part of the resonances  are known when the graph is not unbalanced.

Also, note that one cannot hope to show that resonances belong to a strip $\R + i [-K, - \varepsilon]$ for some $\varepsilon$. Indeed, it was shown in \cite{CdvT} that, for most quantum graphs, there exist resonances with arbitrarily small imaginary parts and arbitrarily large real parts. We refer the reader to \cite{Ing} for references on the delicate issue of resonances on the real axis.

\subsection{A new formula for resonance counting}

The central result of this article is a formula relating the sum of the values of holomorphic functions over the resonances of a quantum graph with a boundary integral involving the matrix $U_\cQ$.

If $-\infty \leq y_1<y_2 \leq +\infty$, we write $\Omega_{y_1,y_2} := \{z\in \C ; \Im z \in (y_1,y_2)\}$. 
We denote by $\mathcal{H}^1(\Omega_{y_1, y_2})$ the set of holomorphic functions on $\Omega_{y_1, y_2}$.

\begin{tcolorbox}
\begin{theoreme}\label{Prop:GaussCount2}
Let $D, n_0\in \N$,  let $0<L_{min}\leq L_{max}$ and let $\mathcal{Q}\in \kQ_{D, n_0, L_{min}, L_{max}}'$ be a finite open quantum graph. Let  $y_1 \leq Y(D, n_0, L_{min})$, let $y_2 \geq 0$, and let $\varepsilon>0$.

For any $g\in L^1(\Omega_{y_1- \varepsilon,y_2+ \varepsilon})\cap \mathcal{H}(\Omega_{y_1-\varepsilon,y_2+ \varepsilon})$, we have

\begin{equation}\label{eq:GaussCount}
 2i\pi \sum_{z\in \mathrm{Res}(\mathcal{Q})} g(z) = - \sum_{j=1,2} \int_\R g( x + i y_j) \mathrm{Tr} \left[ U'_{\cQ}(x + i y_j) \left( \mathrm{Id}- U_{\cQ}(x + iy_j)\right)^{-1} \right] \mathrm{d}x,
 \end{equation}
with the sum in the left-hand side repeated with its multiplicity.
\end{theoreme}
\end{tcolorbox}

\begin{rem}
The integral in the right-hand side of (\ref{eq:GaussCount}) is absolutely convergent, as will be shown in (\ref{eq:c}) below.
\end{rem}

\begin{rem}
The set of functions $L^1(\Omega_{y_1- \varepsilon,y_2+ \varepsilon})\cap \mathcal{H}(\Omega_{y_1-\varepsilon,y_2+ \varepsilon})$ is non-empty: for instance the Gaussian functions $g(z) = e^{-a (z-z_0)^2}$ belong to it for any $a>0$, $z_0\in \C$.
\end{rem}

Theorem \ref{Prop:GaussCount2}, which is a consequence of the residue formula, can be used to give an alternative proof of Theorem \ref{Th:DavPush} for unbalanced graphs, which might be easier to generalize than the original proof, as it does not use the fact that $\mathrm{\det} (\mathrm{Id} - U_\cQ(z))$ is a linear combination of complex exponentials; this will be done in section \ref{sec:AlterWeyl}. More interestingly, it
can be used to obtain lower bounds on the number of resonances in some regions of the complex plane. The following proposition gives an example of such a lower bound, which we don't expect to be sharp, but which is completely explicit.

\begin{tcolorbox}
\begin{proposition}\label{Prop:GaussLowerBound}
Let $D, n_0\in \N$,  let $0<L_{min}\leq L_{max}$ and let $\mathcal{Q}\in \kQ_{D, n_0, L_{min}, L_{max}}'$ be a finite open quantum graph.

Set 
$$a =  \frac{\ln 2}{ \left( 2 \frac{\ln 32}{L_{min}} - Y(D, n_0, L_{min})\right)^2}.$$

Let $x_0\in \R$, and let $\alpha>0$ be such that
\begin{equation}\label{eq:CondAlpha2}
\frac{\alpha}{L_{min}} \geq \left(\left[Y- \frac{\ln 16}{L_{min}}\right]^2 - \frac{1}{a} \ln \left( L_{min} \ln 2 \frac{1-e^{-\frac{a}{L^2_{min}}}}{8 \left( 2 L_{max} \frac{1+ \ln (D+ n_0)}{L_{min}}  + 0.6  \right)} \sqrt{\frac{\pi}{a}} \right) \right)^{1/2}.
\end{equation}

Then we have
$$\mathcal{N}_\cQ \left( \left\{z\in \C ; \Re z - \frac{\alpha}{L_{min}} , \Re z + \frac{\alpha}{L_{min}}  \right\} \right) \geq \frac{\cL_\cQ}{8} \sqrt{\frac{\pi}{a}}.$$
\end{proposition}
\end{tcolorbox}

\begin{rem}
Theorem 3 in \cite{Langer}, along with the discussion in \cite[\S 3]{NonWeyl}  implies that, if $x_1< x_2$ are such that $\cQ$ has no resonance $z$ such that $\Re z = x_1$ or $x_2$, then we have
\begin{equation}\label{eq:EncadrClassique}
 -|B(\cQ)| + \frac{\cL_\cQ}{2\pi} (x_2-x_1) \leq  \mathcal{N}_{\cQ} \left( \left\{ \Re z \in (x_1, x_2) \right\} \right)  \leq |B(\cQ)| + \frac{\cL_\cQ}{2\pi} (x_2-x_1),
\end{equation}
where $|B(\cQ)|$ is the number of oriented edges of the quantum  graph $\cQ$.

Equation (\ref{eq:EncadrClassique}) directly implies  Theorem \ref{Th:DavPush}, and it is very relevant when the graph $\cQ$ is fixed, and $x_2-x_1$ is taken large enough. However, when $|B(\cQ)|$ is large, to obtain a non trivial lower bound on the number of resonances using (\ref{eq:EncadrClassique}), one must take $x_2-x_1$ large, and hence work in large boxes. By contrast, equation (\ref{eq:CondAlpha2}) implies a condition on the size of the boxes under consideration which is hard to express, but is independent of the size of the graph: it depends only on $D, n_0, L_{min}$ and $L_{max}$. 

Hence, Proposition \ref{Prop:GaussLowerBound} is a real improvement over (\ref{eq:EncadrClassique}) when working with large graphs.
\end{rem}

\paragraph{Numerical example}

Suppose we take $n_0=1$, $D=4$, $L_{min}=1$, $L_{max}=2$. We then have $Y=-\ln 5$, $a\approx 9.5\times 10^{-3}$, $\frac{1}{8}\sqrt{\frac{\pi}{a}}\approx 2.3$, and a tedious computation implies that (\ref{eq:CondAlpha2}) can be rephrased as
$$\frac{\alpha}{L_{min}} \geq 26.7.$$

Therefore,  every vertical strip of length at least $54$ contains at least $2.3\times  \cL_\cQ$ resonances.

On the other hand, we know that the number of resonances in $-R \leq \Re z \leq R$ is of the order of $R \frac{2\cL_\cQ}{\pi}$ as $R$ is large.  
Hence,  a vertical strip of length $54$ contains,  on average,  $\frac{54 \cL_{\cQ}}{\pi}\approx 17 \times \cL_\cQ$ resonances.
We thus see that, though our result is far from being sharp, it is only one order of magnitude smaller than the average result.

\subsection{Asymptotic distribution of resonances of large quantum graphs}

As explained in the previous paragraph, Theorem \ref{Prop:GaussCount2} can be used to estimate numbers of resonances, even for large graphs. Actually, it would be desirable to understand the asymptotic spectral properties of sequences of quantum graphs.

Namely, given a sequence of (larger and larger) quantum graphs $(\cQ_N)$, can one give an asymptotic formula for the number of resonances in some given region of the complex plane (independent of $N$)? This kind of problematic was first raised in \cite{Ing} for open quantum graphs, but similar results existed for the spectrum of closed quantum graphs (see \cite{BSQG}). There has also been a large interest in the asymptotic properties of the eigenfunctions of large quantum graphs: see \cite{QEQG} and the references therein.

When studying the spectral asymptotics of large quantum graphs, a natural assumption is that of \emph{Benjamini-Schramm convergence}, whose precise definition we recall in section \ref{sec:BS}, and which can be informally described as follows. Around a vertex $v_0$ in a quantum graph, one can define several ``local" quantities: for instance, the number of neighbours of $v_0$, the average length of the edges attached to $v_0$, the number of cycles in a ball of radius 10 around $v_0$...
One may then take $v_0$ uniformly at random in the graph, so as to make statistics of these local quantities. We say that a sequence of quantum graphs converges in the sense of Benjamini-Schramm if all these statistics of local quantities converge. The Benjamini-Schramm limit is then a probability measure on the set of rooted quantum graphs $\mathrm{ROQ}$.

The appeal for Benjamini-Schramm convergence comes from its compactness properties (see Lemma \ref{lem:QCompact}): given a sequence of quantum graphs with uniformly bounded data, one can always extract a subsequence converging in the sense of Benjamini-Schramm. Of course, Benjamini-Schramm convergence deals with local quantities, and spectral quantities have no reason to be local, since eigenfunctions can be delocalized all over the graph. Still, we make the following conjecture about the convergence of the empirical spectral measures. These are locally finite Borel measures on $\C$, given by
$$\mu_\mathcal{Q}:= \frac{1}{\cL_\cQ} \sum_{z\in \mathrm{Res}(\cQ)} \delta_{z},$$
the sum being repeated with the multiplicity of the resonances.

\begin{tcolorbox}
\begin{Conj}
Let $D, n_0\in \N$,  and let $0<L_{min}\leq L_{max}$. Let $\mathbb{P}\in \mathcal{P}(\mathrm{ROQ})$. Then there exists a locally finite Borel measure $\mu_{\mathbb{P}}$ on $\C$, 
such that the following holds.

If $(\cQ_N)$ is a sequence of quantum graphs belonging to $\mathfrak{Q}_{D, n_0, L_{min}, L_{max}}$ which converges in the sense of Benjamini-Schramm to $\mathbb{P}$, then $(\mu_{\cQ_N})$ converges vaguely to $\mu_{\mathbb{P}}$. In other words, for $\chi \in C_c(\C)$, we have
$$\sum_{z\in \mathrm{Res}(\cQ_N)} \chi(z) \longrightarrow \int_\C \chi(z) \mathrm{d} \mu_{\mathbb{P}}(z).$$
\end{Conj}
\end{tcolorbox}

This conjecture was proven in  \cite{Ing} for unbalanced quantum graphs, in the case where $\mathbb{P}$ is supported on the set of closed quantum graphs, i.e., of quantum graphs such that $\boldsymbol{n} \equiv 0$. Note that the end of the proof of \cite[Theorem 2]{Ing} can be considerably simplified using our Theorem \ref{Prop:GaussCount2}. Previously, the analogue of the conjecture was proven in \cite{BSQG} for closed quantum graph in a very general setting (including potentials on the edges, and general coupling conditions).

Note that, thanks to Prokhorov's theorem and to upper bounds on the number of resonances in a bounded set (see \eqref{eq:SuperJensen} below, which is a consequence of Jensen's formula), up to extracting a subsequence, $(\mu_{\mathcal{Q}_N})$ converges vaguely; the statement of the conjecture is thus that all the accumulation points for the vague topology are the same, and that they only depend on $\mathbb{P}$. To prove this, it would be sufficient to prove that, for any $\chi \in C_c^\infty(\C)$, we have
\begin{equation}\label{eq:ConjLimit}
\sum_{z\in \mathrm{Res}(\cQ_N)} \chi(z) \longrightarrow \ell_\chi,
\end{equation}
and that the limit $\ell_\chi$ depends only on $\mathbb{P}$ and $\chi$.

We are going to show (\ref{eq:ConjLimit}) for any limit $\mathbb{P}$ (but still only for unbalanced graphs), but only for a special family of functions $\chi$, that are holomorphic in a strip.

\begin{tcolorbox}
\begin{theoreme}\label{Th:BS}
Let $D, n_0\in \N$,  and let $0<L_{min}\leq L_{max}$. Let $(\cQ_N)$ be a sequence of quantum graphs belonging to $\mathfrak{Q}'_{D, n_0, L_{min}, L_{max}}$. Suppose that $(\cQ_N)$ converges in the sense of Benjamini-Schram to some measure $\mathbb{P}$. 

There exist functions $\Lambda_\mathbb{P} : \Omega_{-\infty, Y} \cup \Omega_{0, +\infty} \longrightarrow \C$ such that the following holds:
\begin{itemize}
\item for any $ y\in (-\infty, Y)\cup (0, +\infty)$, we have $|\Lambda_\mathbb{P}(x+iy)| \leq M(y)$ for some $M(y)$ depending on $y$ and $D, n_0, L_{min}, L_{max}$, but not on $x$.
\item If $y_1 <Y(D, n_0, L_{min}, L_{max})$, $y_2>0$, and $\varepsilon>0$ and if $g\in L^1(\Omega_{y_1- \varepsilon,y_2+ \varepsilon})\cap \mathcal{H}(\Omega_{y_1-\varepsilon,y_2+ \varepsilon})$, we have
\begin{equation}
\langle \mu_{\cQ_N}, g \rangle \underset{N\to \infty}{\longrightarrow} \sum_{j=1,2} \int_\R g(x + iy_j) \Lambda_\mathbb{P}(x+iy_j)\mathrm{d}x.
\end{equation}
\end{itemize}
\end{theoreme}
\end{tcolorbox}

\paragraph{Organization of the paper}
In section \ref{Sec:Prelim}, we will recall the definition of open quantum graphs, of their resonances, and of the secular equation. Section \ref{sec:Proof} will be devoted to the proof of Theorem \ref{Prop:GaussCount2}. In section \ref{sec:Consequences}, we will use Theorem \ref{Prop:GaussCount2} to prove Proposition \ref{Prop:GaussLowerBound}. We will also give an alternative proof of Theorem \ref{Th:BS} for unbalanced graphs. Finally, in section \ref{sec:BS}, we will recall the definition of Benjamini-Schramm convergence for open quantum graphs, and we will prove Theorem \ref{Th:BS}.

\paragraph{Acknowledgements} The author was partially funded by the Agence Nationale de la Recherche, through the project ADYCT (ANR-20-CE40-0017).

\section{Open quantum graphs and their resonances} \label{Sec:Prelim}
\subsection{Definition of open quantum graphs }
An (open) quantum graph is given by a finite graph, where each edge is given a length and where infinite edges (called \emph{leads}) are attached to some of the vertices. More precisely
\begin{tcolorbox}
\begin{definition}
A \emph{quantum graph} $\mathcal{Q}=(V,E,L, \mathbf{n})$ is the data of
\begin{itemize}
\item A graph $G=(V,E)$  with vertex set $V$ and edge set $E$.
\item A map $L: E\rightarrow (0,\infty)$.
\item A map $\mathbf{n}: V \longrightarrow \N\cup \{0\}$.
\end{itemize}

The quantum graph will be called \emph{finite} if $G$ is a finite graph. 
\end{definition}
\end{tcolorbox}

The graph $(V,E)$ should be thought of as the compact part of our graph, the map $L$ gives the length of the edges in the compact part, while the map $\mathbf{n}$ gives the number of infinite edges attached to each vertex. If $v\in V$, we denote by $\textcolor{black}{d(v)}$ the (internal) degree of $v$, i.e., the number of $e\in E$ to which $v$ belongs. We define the total length of the (internal part of the) graph by
\begin{equation}\label{eq:LongueurTotale}
\mathcal{L}_\mathcal{Q}:= \sum_{e\in E} L(e).
\end{equation}

We let $B= B(\cQ)$ be the set of oriented internal edges (or bonds) associated to $E$. If $b\in B$, we shall denote by $\hat{b}$ the reverse bond. We write $o_b$ for the origin of $b$ and $t_b$ for the terminus of $b$.  We will also write $L_b$ for the length of the edge to which $b$ is associated.

We define the external bonds by $B_{ext} := \bigsqcup_{v\in V} \bigsqcup _{k=1}^{\mathbf{n}(v)}\textcolor{black}{ \{}(v, k) \textcolor{black}{\}}$, and we write $o_b= v$ if $b=(v,k)\in B_{ext}$, and set $\hat{B}= B\cup B_{ext}$.

In the sequel, we will always consider graphs with bounded data, as in the following definition:

\begin{tcolorbox}
\begin{definition}
Let $D, n_0\in \N$,  and let $0<L_{min}\leq L_{max}$.
We denote by $\mathfrak{Q}_{D, n_0, L_{min}, L_{max}}$ the set of open quantum graphs such that we have
\begin{align*}
& \forall v\in V, \textcolor{black}{d(v)}\leq D \text{ and } \mathbf{n}(v)  \le n_0\\
&\forall e\in E, L_{min} \le L(e)\le L_{max}.
\end{align*}

We will denote by $\kQ'_{D, n_0, L_{min}, L_{max}}$ the set of quantum graphs in $\kQ_{D, n_0, L_{min}, L_{max}}$ such that
\begin{equation}\label{eq:Unbalanced}
\forall v\in V, \mathbf{n}(v) \neq d(v).
\end{equation}
\end{definition}
\end{tcolorbox}



\subsection{Scattering resonances of open quantum graphs}\label{sec:Res}

A $C^2$ function on the graph will be a collection of maps $f=(f_b)_{b\in \hat{B}}$, with $f_b\in C^2([0,L_b])$ if $b\in B$, $f_b\in C^2([0,\infty))$ if $b\in B_{ext}$, and such that
\begin{equation} \label{eq:CondSym}
\forall b\in B, f_b (\cdot) = f_{\hat{b}}(L_b - \cdot).
\end{equation}
 We will write $C^2(\mathcal{Q})$ for the set of such functions. We may now define the boundary conditions we put at the vertices.

\begin{tcolorbox}
\begin{definition}\label{def:Kirch}
We say that $f\in C^2(\mathcal{Q})$ satisfies \emph{Kirchhoff boundary conditions} if it satisfies
\begin{itemize}
\item \textbf{Continuity}: For all $b,b'\in \hat{B}$, we have $f_b(0) = f_{b'}(0) =:f(v)$ if $o_b=o_{b'}=v$.\\
\item \textbf{Current conservation:} For all $v\in V$, 
\begin{equation*}
\sum_{b\in \hat{B}:o_b=v} f_b'(0)= \textcolor{black}{0}\,.
\end{equation*}
\end{itemize}
\end{definition}
\end{tcolorbox}

Let us now move to the definition of scattering resonances of quantum graphs.

\begin{tcolorbox}
\begin{definition}\label{def:Res}
Let $\cQ$ be a finite quantum graph.
A number $z\in \C$ is called a \emph{scattering resonance} of $\cQ$ if there exists $f\in C^2(\cQ)$ such that
\begin{enumerate}
\item $f$ satisfies the Kirchhoff boundary conditions.
\item For all $b\in \hat{B}$, we have $-f_b''= z^2 f_b$.
\item For all $b\in B_{ext}$, we have $f_b(x) = f_b(0) e^{i zx}$.
\end{enumerate}


We will write
$$\mathrm{Res}(\mathcal{Q}):= \{ \text{Resonances of } \mathcal{Q}\} \subset \C.$$

\end{definition}
\end{tcolorbox}


Note that, if we did not impose the last condition (and if $\mathbf{n}$ is not identically zero), then any number $z\in \C$ would be a scattering resonance. On the other hand condition 3. imposes that scattering resonances must have negative imaginary part. Otherwise, the function $f$ would be an eigenfunction of a selfadjoint operator, associated to a non-real eigenvalue.

\subsubsection{The secular equation for scattering resonances}\label{sec:Secular}

If $b, b'\in B(\cQ)$, we define the quantity
$$\sigma_{b,b'} = \begin{cases} \frac{2}{\boldsymbol{n}(v)+d(v)} &\text{ if } o_b=o_{b'}=v \text{ and } b'\neq b\\
\frac{2}{\boldsymbol{n}(v)+d(v)} - 1 &\text{ if } b'=b \text{ with } o_b=v\\
0 &\text{ if } o_b\neq o_{b'}. \end{cases}$$

We then define the matrices $D_\cQ(z)$, $S_\cQ$ and $U_\cQ(z)$ whose lines and columns are indexed by the elements of $B(\cQ)$ by
\begin{equation}\label{eq:DefU}
\begin{aligned}
D_\cQ(z)_{b,b'}&= \delta_{b,b'} \ee^{iz L_b}\\
(S_\cQ)_{b,b'} &= \sigma_{b,\hat{b'}}\\
U_\cQ(z) &= S_\cQ D_\cQ(z).
\end{aligned}
\end{equation}


As explained in the introduction, these matrices give a characterization of scattering resonances:
\begin{equation*}
\forall z\in \C \setminus \{0\},~~~~z \text{ is a resonance } \Longleftrightarrow \mathrm{\det} (\mathrm{Id} - U_\cQ(z)) = 0.
\end{equation*}

We refer the reader to \cite[\S 3.2]{Ing} for a proof of the previous equation.

The multiplicity of a resonance $z_0$ will be defined the order of the zero $z_0$ of the holomorphic function $z\mapsto \mathrm{\det} (\mathrm{Id} - U_\cQ(z)) $. We refer to \cite{NonWeyl} for a proof of the fact that this definition of multiplicity coincides with the  other natural definitions when $z\neq 0$.   Note that, when $z=0$, the various definitions of multiplicity need not coincide.

\subsubsection{Proof of \eqref{eq:ResonancesInStrip}}\label{seq:ProofStrip}
If $\cQ\in \mathfrak{Q}'_{D, n_0, L_{min}, L_{max}}$, we see that the matrices $\sigma^{(v)}$ are all invertible, and we have 
$$\|(\sigma^{(v)})^{-1}\| = \frac{\mathbf{n}(v)+ d(v)}{|d(v) - \mathbf{n}(v)|} \leq \mathbf{n}(v)+ d(v)\leq n_0 +D.$$

If $J$ is the $B\times B$ matrix such that $J_{b,b'} = \delta_{b', \widehat{b}}$, then 
the matrix $S_\cQ J$ is a block matrix with blocks $\sigma^{(v)}$, so it can be inverted block by block. We deduce that $\|S_\cQ^{-1}\|\leq n_0 + D$. In particular, we see that $U_\cQ(z)$ is invertible, and that for any $z\in \C^-$, we have

\begin{equation}\label{eq:BorneInverseU}
\|U_\cQ(z)^{-1}\| \leq (n_0 + D) e^{\Im z L_{min}}.
\end{equation}

Recalling that $\mathrm{\det} (\mathrm{Id} - U_\cQ(z)) = 0$ if and only if $\mathrm{\det} (\mathrm{Id} - U_\cQ^{-1}(z)) = 0$, we deduce (\ref{eq:ResonancesInStrip}).

\section{Proof of Theorem \ref{Prop:GaussCount2}}\label{sec:Proof}

In all the proof, we shall write $f(z) := \det ( \mathrm{Id}- U_{\cQ}(z))$, which is a holomorphic function on $\C$.

If $\Omega\subset \C$, we will denote by $\mathrm{Res}_\cQ (\Omega) := \mathrm{Res}(\cQ) \cap \Omega$, i.e., the set of resonances of $\cQ$ in $\Omega$.

\subsection{Preliminaries}

\subsubsection{Of determinants and traces}
First of all, writing, for any $z\in \C$ where $f$ does not vanish,
\begin{align*}
\det ( \mathrm{Id}- U_{\cQ}(z+z')) &= \det \left( \mathrm{Id}- U_{\cQ}(z) - z' U'_\cQ(z) + o(z')\right) \\
&=\det \left( \mathrm{Id}- U_{\cQ}(z)\right) \det \left(\mathrm{Id} - z' (\mathrm{Id}- U_{\cQ}(z))^{-1} U'_{\cQ}(z) + o(z')\right) \\
&= f(z) \left( 1 - z' \mathrm{Tr}\left[ U_{\cQ}'(z) (\mathrm{Id}- U_{\cQ}(z))^{-1} \right] + o(z)\right),
\end{align*}

 we obtain that
\begin{equation}\label{eq:DerivDet}
\begin{aligned}
\frac{f'(z)}{f(z)} &= - \mathrm{Tr}\left[ U_{\cQ}'(z) (\mathrm{Id}- U_{\cQ}(z))^{-1} \right]. \\
\end{aligned}
\end{equation}

Using the definition of $U(z)$ and the circularity of the trace, this can be rewritten as
\begin{equation}\label{eq:TraceFormula}
\begin{aligned}
\frac{f'(z)}{f(z)} &= - \mathrm{Tr}\left[U_{\cQ}'(z) (\mathrm{Id}- U_{\cQ}(z))^{-1} \right]\\
&=- \mathrm{Tr}\left[ S_{\cQ} D_{\cQ}(z)iL_\cQ (\mathrm{Id}- U_{\cQ}(z))^{-1} \right]\\
&=- \mathrm{Tr}\left[ (\mathrm{Id}- U_{\cQ}(z))^{-1} U_{\cQ}(z) iL_\cQ\right]\\
&=  i \mathcal{L}_{\cQ} - \mathrm{Tr} \left[  (\mathrm{Id}- U_{\cQ}(z))^{-1} \right],
\end{aligned}
\end{equation}
since $\mathrm{Tr}(L_\cQ) = 2 \cL_\cQ$.

\subsubsection{Reminder on the trace norm}

To estimate $f$, we will often use several norms on matrices, whose definition we now recall.

If $A$ is a $d\times d$ matrix, we shall denote by $\|A\|$ its operator norm, i.e.
$$\|A\| := \sup_{x\in \C^d \setminus \{0\}} \frac{\|Ax \|}{\|x\|}.$$

Its \emph{trace norm} is defined by 
$$\|A\|_1 := \sum_{j=1}^N \sigma_j(A),$$
where the $\sigma_j(A)$ are the singular values of $A$, i.e., the eigenvalues of $A^*A$.

The following properties of the trace norm, which are standard, will be useful in the proof:
\begin{equation}\label{eq:SchattenDim}
\|A\|_1 \leq d \|A\|,
\end{equation}
\begin{equation}\label{eq:PropSchatten}
|\mathrm{\det} (\mathrm{Id} + A)| \leq e^{\|A\|_1},
\end{equation}
\begin{equation}
\left|\mathrm{Tr}[A]\right| \leq \|A\|_1.
\end{equation}

If $A, B\in \mathcal{M}_d(\C)$, we have
\begin{equation}\label{eq:PropSchatten1}
\begin{aligned}
\|AB \|_1 &\leq \|A\| \|B\|_1,\\
\|BA \|_1 &\leq \|A\| \|B\|_1.
\end{aligned}
\end{equation}

\subsubsection{Preliminary estimates on $f$}

\paragraph{Upper bounds on $f$}
Recall that we write $f(z) := \det ( \mathrm{Id}- U_\cQ(z))$. First of all, let us note that

\begin{equation}\label{eq:BorneU}
\|U_\cQ(z)\| \leq \|S_{\cQ}\| \|e^{izL_\cQ}\| \leq \|e^{izL_\cQ}\| \leq \begin{cases} 
e^{- \Im z L_{min}} \text{ when } \Im z >0\\
e^{|\Im z| L_{max}} \text{ when } \Im z \leq 0
\end{cases}
\end{equation}

Therefore, thanks to (\ref{eq:SchattenDim}) and (\ref{eq:PropSchatten}), we have
\begin{equation}\label{eq:BorneFGenerale}
|f(z)|\leq e^{\|U_{\cQ}(z)\|_1} \leq \exp\left[ |B(\cQ)| \|U_\cQ(z)\| \right] \leq \exp \left[ |B(\cQ)|  e^{ |B(\cQ)| L_{max} \max(0, -\Im z)} \right].
\end{equation}
In particular, this quantity is independent of $\Re z$. Hence, for any $y_1\in \R$, there exists $C(y_1, \cQ)$ such that
\begin{equation}\label{eq:Bornef}
|f(z)| \leq C(y_1, \cQ) ~~\forall z \in \Omega_{y_1,+\infty}.
\end{equation}

The function $f$ being holomorphic, the Cauchy formula implies that there exists also a constant $C'(y_1, \cQ)$
such that
\begin{equation}
\label{eq:BorneDerivf}
|f'(z)| \leq C'(y_1, \cQ) ~~\forall z \in \Omega_{y_1,+\infty}.
\end{equation}

\paragraph{Lower bounds on $f(z)$ when $\Im z >0$}
Equation (\ref{eq:BorneU}) implies that  if $\Im z>0$, then $\|U_\cQ(z)\| <1$, so that $(\mathrm{Id} - U_\cQ(z))$ can be inverted by a Neumann series as
\begin{equation}\label{eq:InvDessus}
\begin{aligned}
(\mathrm{Id} - U_\cQ(z))^{-1} = \mathrm{Id} + \sum_{k=1}^{+\infty} (U_\cQ(z))^k = \mathrm{Id} + R,
\end{aligned}
\end{equation}
with 
\begin{equation}\label{eq:ControleR}
\|R\| \leq  \frac{\|U_\cQ(z)\|}{1- \|U_\cQ(z)\|}\leq \frac{e^{-\Im z L_{min}}}{1- e^{-\Im z L_{min}}}.
\end{equation}


In particular, we have, when $\Im z>0$
\begin{equation}\label{eq:PasDidee}
\begin{aligned}
 \frac{1}{|f(z)|} &=\left| \det\left((\mathrm{Id} - U_\cQ(z))^{-1} \right)\right| \\
 &=  \det \left( \mathrm{Id} +R \right)\\
 &\leq e^{|B(\cQ)| \|R\|},
 \end{aligned}
 \end{equation}
 which is independent of $\Re z$.

Therefore, we see that, for any $y_1>0$, there exists $C_0(y_1, \cQ)$ such that
\begin{equation}\label{eq:Upper}
\frac{1}{|f(z)|} \leq C_0(y_1, \cQ) ~~\forall z \in \Omega_{y_1,+\infty}.
\end{equation}

\paragraph{Lower bounds on $f(z)$ when $\Im z <Y$}

Recalling \eqref{eq:BorneInverseU}, we see that if $y_0 < Y$, there exists $c_0<1$ such that
$$\|U_\cQ(z)^{-1}\| \leq c_0 ~~ \text{ for all } z\in \C \text{ with } \Im z< y_0.$$

Hence, when $\Im z < y_0$,
\begin{equation}\label{eq:InvDessous}
(\mathrm{Id} - U_\cQ(z))^{-1} = -U_\cQ(z)^{-1}  (\mathrm{Id} - U_\cQ(z)^{-1})^{-1} =  \sum_{k=1}^{+\infty} (U_\cQ(z))^{-k},
\end{equation}
so that there exists $c_1>0$ such that
\begin{equation}\label{eq:BorneInvDessous}
\left\|(\mathrm{Id} - U_\cQ(z))^{-1} \right\| \leq c_1 ~~ \text{ for all } z\in \Omega_{-\infty, y_0}.
\end{equation}

In particular, equations (\ref{eq:InvDessus}), (\ref{eq:ControleR}) and (\ref{eq:BorneInvDessous}) along with (\ref{eq:DerivDet}) imply that, if $y_0<Y$ and $y_1>0$,  there exists $c_1= c_1(y_0, y_1, \cQ)>0$ such that
\begin{equation}\label{eq:BorneFprimesurF}
\left| \frac{f'(z)}{f(z)}\right| \leq c_1 ~~~~ \forall z\in \Omega_{-\infty, y_0} \cup \Omega_{y_1, +\infty}.
\end{equation}

\subsection{Complex analysis}

If $g$ is a holomorphic function, and if $\Omega$ is a subset of $\C$ such that $f$ does not vanish on $\partial \Omega$, then 
\begin{equation}\label{eq:Residue}
\int_{\partial\Omega} g(z) \frac{f'(z)}{f(z)} \mathrm{d}z= 2i\pi \sum_{z\in \mathrm{Res}_\cQ(\Omega)} g(z),
\end{equation}
where the integral in the left-hand side is performed clockwise, while the sum in the right-hand side is over the zeroes of $f$, repeated with multiplicity.

This is simply an application of the residue formula, noting that if $f$ has a zero of order $m$ at some point $z_0$, then $g(z) \frac{f'(z)}{f(z)}$ can be written as $m \frac{g(z_0)}{z-z_0}$ plus  a function which is holomorphic in a neighbourhood of $z_0$.


We apply equation (\ref{eq:Residue}) with $\Omega=\Omega_{x^-, x^+ ,y_1,y_2}= [x^-, x^+] + i [y_1, y_2]$, for some $y_1<  Y$ and $y_2 >0$. We take $x^- < 0 < x^+$ so that $f$ does not vanish on the horizontal sides of $\partial \Omega_{x^-, x^+ ,y_1,y_2}$.

Since $f$ has isolated zeroes, for almost every $x^-, x^+$, $f$ does not vanish on $x^\pm + i [y_1, y_2]$. Therefore, we have

\begin{equation}\label{eq:Contour}
\begin{aligned}
2i\pi \sum_{z\in \mathrm{Res}_\cQ(\Omega_{x^-, x^+,y_1,y_2})} g(z) &= - \sum_{j=1,2} \int_{x^-}^{x^+} g( x + i y_j) \mathrm{Tr} \left[ U'_{\cQ}(x + i y_j) \left( \mathrm{Id}- U_{\cQ}(x + iy_j)\right)^{-1} \right] \mathrm{d}x\\
& - i\sum_\pm \int_{y_1}^{y_2} g( x^\pm + i y) \mathrm{Tr} \left[ U'_{\cQ}(x^\pm + i y) \left( \mathrm{Id}- U_{\cQ}(x^\pm + i y)\right)^{-1} \right] \mathrm{d}y.
\end{aligned}
\end{equation}

\subsection{Lower bounds in strips}

We would like to take the limit $x^\pm \to \pm \infty$ in (\ref{eq:Contour}), and to show that the lateral terms (corresponding to the last two integrals in the right-hand side) tend to zero. To this end, we will need some more precise lower bounds on $f$ on $x^\pm+ i [y_1, y_2]$. This will be given by the following Lemma, proven in \cite[Lemma 5]{Ing}.

\begin{tcolorbox}
\begin{lemme}\label{lem:GoodPoint}
Let $z_0\in \C$, let $s,t>0$, and let $I_1 \subset [\Re z_0 -  s, \Re z_0 + s]$  be an interval.
There exists a constant $C=C(s,t)$ such that the following holds.

For any $f$ holomorphic function on $\C$, we may find $x\in I_1$,  such that, for all $y' \textcolor{black}{\in} [\Im z_0 - t, \Im z_0 + t]$, we have
\begin{equation}\label{eq:Lem}
\begin{aligned}
\left| \ln |f(x+iy')| \right| \leq C' \left( \ln \sup_{z\in D(z_0, 3 \max(s,t))} |f(z)| + \left| \ln |f(z_0)| \right| + \left| \ln |I_1| \right|  \right).
\end{aligned}
\end{equation}
\end{lemme}
\end{tcolorbox}

Let $n\in \N$. We may apply Lemma \ref{lem:GoodPoint} with $z_0 = \pm n+i$, $t= 1+ y_2 +|y_1|$, $s=1$ and $I_1 = [n, n + 1]$ or $[-n-1, -n]$. Estimating the right-hand side of (\ref{eq:Lem}) using equations (\ref{eq:Bornef}) and (\ref{eq:Upper}), we obtain that there exists $x_n^+\in [n, n+1]$ and $x_n^-\in [-n-1, -n]$ such that, for all $y\in [y_1, y_2]$, we have
\begin{equation}\label{eq:LowerStrip}
 |f( x_n^\pm + iy)| \geq c_2(\cQ, y_1, y_2)~~ \forall y \in [y_1, y_2].
 \end{equation}

Note in particular  that $ c_2(\cQ, y_1, y_2)$ does not depend on $n$.


\subsection{End of the proof}
Applying (\ref{eq:Contour}) with $x^\pm = x_n^\pm$, we will obtain (\ref{eq:GaussCount}) if we can show the following three points:
\begin{equation}\label{eq:a}
\int_{y_1}^{y_2} g( x_n^\pm + i y) \frac{f' (x_n^\pm + i y)}{f (x_n^\pm + i y)} \mathrm{d}y  \underset{n\to \infty}{\longrightarrow} 0,
\end{equation}
\begin{equation}\label{eq:b}
\sum_{z\in \mathrm{Res}_\cQ(\C\setminus \Omega_{x_n^-, x_n^+, y_1, y_2})} g(z) \underset{n\to \infty}{\longrightarrow} 0,
\end{equation}
and
\begin{equation}
\begin{aligned}\label{eq:c}
\int_\R  \left| g( x + i y_j) \mathrm{Tr} \left[ U'_{\cQ}(x + i y_j) \left( \mathrm{Id}- U_{\cQ}(x + iy_j)\right)^{-1} \right] \right| \mathrm{d}x < +\infty.
\end{aligned}
\end{equation}

Equation (\ref{eq:c}) implies that the right-hand term in (\ref{eq:GaussCount}) is well-defined, and that the first term in the right-hand side of (\ref{eq:Contour}) converges to it.

\paragraph*{Proof of (\ref{eq:a})}

Thanks to (\ref{eq:LowerStrip}) and (\ref{eq:BorneDerivf}), to obtain (\ref{eq:a}), it is sufficient to show that
\begin{equation*}
\int_{y_1}^{y_2} |g( x_n^\pm + i y)|  \mathrm{d}y  \underset{n\to \infty}{\longrightarrow} 0,
\end{equation*}

Let $\varepsilon>0$ be such that $Y - y_1 > \varepsilon$ and $y_2> \varepsilon$.

Thanks to the mean-value property, we know that
$$g(x_n^\pm + iy) = \dashint_{B(x_n^\pm)} g(z) \mathrm{d} z, $$
so that

$$\sum_n \int_{y_1}^{y_2} |g( x_n^\pm + i y)| \leq  \int_{y_1}^{y_2} \sum_n \dashint_{B(x_n^\pm)} |g(z)| \mathrm{d} z \leq \int_{\Omega_{y_1, y_2}}  |g(z)| \mathrm{d} z.$$

Since this integral is convergent, the series in the right-hand term converges, and we deduce (\ref{eq:a}).

\paragraph*{Proof of (\ref{eq:b})}

First of all, we claim that for any $r>0$, there exists $C(r)>0$ such that
\begin{equation}\label{eq:UpperNombreRes}
\forall z\in \Omega_{y_1,y_2},  \mathcal{N}_\cQ (B(z, r)) \leq C(r).
\end{equation}

Indeed, any ball $B(z, r)$ may be included in a ball $B(z',r')$ with $\Im z' =1$ and $r' \leq r + |Y| +1$; it is thus sufficient to bound $\mathcal{N}_\cQ (B(z', r'))$. This can be done using Jensen's formula.  Recall that if $n(t)$ denotes the number of zeroes $f(z)$ such that $|z-z'|<t$, Jensen's formula tells us that
\begin{equation*}
\int_0^{r '}\frac{n(t)}{t} \mathrm{d}t + \ln |f(z')| = \frac{1}{2\pi} \int_0^{2\pi} \ln |f(z'+e^{i\theta} r')| \mathrm{d}\theta,
\end{equation*}
so that
\begin{equation}\label{eq:Jensen2}
\begin{aligned}
n(r') &\leq \frac{1}{\ln 2}\int_{r'}^{2r'} \frac{n(t)}{t} \mathrm{d}t\\
 &\leq \frac{1}{\ln 2} \left( \ln \max_{|z-z'|= 2r'} |f(z)| - \ln |f(z')| \right),
\end{aligned}
\end{equation}
and the right-hand side can be bounded using (\ref{eq:LowerStrip}) and (\ref{eq:BorneDerivf}). Equation (\ref{eq:UpperNombreRes}) follows.

Thanks to the mean value property, we have
\begin{align*}
\sum_{z\in \mathrm{Res}_\cQ(\C \setminus \Omega_{x_n^-, x_n^+, y_1, y_2})} |g(z)| \leq \sum_{z\in \mathrm{Res}_\cQ(\C \setminus \Omega_{x_n^-, x_n^+, y_1, y_2})} \frac{1}{\pi \varepsilon^2} \int_{D(z, \varepsilon)} |g(z')| \mathrm{d}z'.
\end{align*}

Now, all the disks $D(z, \varepsilon)$ in the previous expression are included in $\{ z'\in \Omega_{y_1, y_2} ; |\Re z'| > n-\varepsilon\}$, and thanks to (\ref{eq:UpperNombreRes}), each $z'$ belongs to at most $C(\varepsilon)$ such disks. Therefore, we have
\begin{align*}
\sum_{z\in \mathrm{Res}_\cQ(\C \setminus \Omega_{x_n^-, x_n^+, y_1, y_2})} |g(z)| \leq \frac{C(\varepsilon)}{\pi \varepsilon^2} \int_{\{ z'\in \Omega_{y_1, y_2} ; |\Re z'| > n-\varepsilon\}} |g'(z)| \mathrm{d}z',
\end{align*}
and this quantity goes to zero as $n\to \infty$, since $g\in L^1$.

\paragraph*{Proof of (\ref{eq:c})}

Thanks to (\ref{eq:BorneFprimesurF}), we see that there exists $C>0$ such that
\begin{equation*}
\int_\R  \left| g( x + i y_j) \mathrm{Tr} \left[ U'_{\cQ}(x + i y_j) \left( \mathrm{Id}- U_{\cQ}(x + iy_j)\right)^{-1} \right] \right| \mathrm{d}x \leq C \int_\R  \left| g( x + i y_j) \right| \mathrm{d}x.
\end{equation*}

Now, using the mean value property, we have
\begin{equation}\label{eq:GIntegrable}
\int_\R  \left| g( x + i y_j) \right| \mathrm{d}x \leq \int_\R  \dashint_{B(x+i y_j, \varepsilon)}  \left| g(z)   \right| \mathrm{d}z \mathrm{d}x \leq \frac{1}{\pi \varepsilon^2} \int_{\Omega_{y_1-\varepsilon, y_2+\varepsilon}} |g(z)| \mathrm{d}z, 
\end{equation}
and the result follows.

\section{Consequences of Theorem \ref{Prop:GaussCount2}}\label{sec:Consequences}

In this section, we will use Theorem \ref{Prop:GaussCount2} and its proof to give an alternative proof of Theorem \ref{Th:BS} in the case of unbalanced graphs, and we will prove Proposition \ref{Prop:GaussLowerBound}. To this end, we must first obtain asymptotics for $\frac{f'}{f}$, with $f(z) := \det ( \mathrm{Id}- U_{\cQ}(z))$ as in the previous section.

\subsection{Asymptotics for $\frac{f'}{f}$}

When $\Im z>0$, we may use the second equality in (\ref{eq:TraceFormula}) along with (\ref{eq:ControleR}) to obtain

\begin{equation}\label{fPrimesurfPos}
\begin{aligned}
\left| \frac{f'(z)}{f(z)}\right| &= \left| \mathrm{Tr}\left[ (\mathrm{Id}- U_\cQ(z))^{-1} U_\cQ(z) iL_\cQ\right] \right|\\
&\leq \|L_\cQ\|_1 \sum_{k=1}^{+\infty}\|U_\cQ(z)^k\| \leq \cL_\cQ \frac{\|U_\cQ(z)\|}{1- \|U_\cQ(z)\|}\\
&\leq  2\cL_\cQ \frac{e^{-\Im z L_{min}}}{1- e^{-\Im z L_{min}}}.
\end{aligned}
\end{equation}

On the other hand, when $\Im z< Y$, we have $\|U_\cQ(z)^{-1}\|  <1$, so that
\begin{equation}\label{fPrimesurfNeg}
\begin{aligned}
\frac{f'(z)}{f(z)} &= - \mathrm{Tr}\left[ (\mathrm{Id}- U_\cQ(z))^{-1} U_\cQ(z) iL_\cQ\right]\\
&=   \mathrm{Tr}\left[ iL_\cQ -(\mathrm{Id}- U_\cQ^{-1}(z))^{-1} iL_\cQ\right]\\
&=  2i \mathcal{L}_{\cQ} -i  \sum_{k=1}^{+\infty} \mathrm{Tr}\left[ U_\cQ^{-k}(z) L_\cQ\right].
\end{aligned}
\end{equation}

Thanks to (\ref{eq:BorneInverseU}), the sum in the last equality has its modulus bounded by 
\begin{equation}\label{eq:fPrimeSurfReste}
\begin{aligned}
2\mathcal{L}_\cQ \sum_{k=1}^{+\infty} \|U_\cQ(z)^{-1}\|^k = 2 \mathcal{L}_\cQ \frac{\|U_\cQ(z)^{-1}\|}{1- \|U_\cQ(z)^{-1}\|} \leq 2 \mathcal{L}_\cQ \frac{e^{(\Im z-Y) L_{min}}}{1- e^{(\Im z-Y) L_{min}}}.
\end{aligned}
\end{equation}

\subsection{An alternative proof of Weyl's law for unbalanced graphs}\label{sec:AlterWeyl}
We apply equation (\ref{eq:Contour}) with $x^- = x_n^-$, $x^+ = x_n^+$, $y_1 = -\sqrt{n}$, $y_2= \sqrt{n}$ and $g$ the function constant equal to one. Hence, the right-hand side is exactly
$$2i\pi \mathcal{N}_\cQ \left( \left\{ \Re z \in [x_n^n, x_n^+]\right\} \right),$$
and thanks to (\ref{eq:UpperNombreRes}), this is $2i\pi  \mathcal{N}_{\cQ} (D(0,n)) + O(1)$.

To bound the lateral integrals, we use (\ref{eq:LowerStrip}) and (\ref{eq:BorneDerivf}) when $\Im z\in [Y-1, 1]$, while when $\Im z \in [- \sqrt{n}, Y-1] \cup [1, \sqrt{n}]$, we use \eqref{fPrimesurfPos} and \eqref{eq:fPrimeSurfReste}. 
 We deduce that the lateral integrals are $O(1)$. 

Therefore, we have
\begin{equation}
\frac{1}{n} \mathcal{N}_{\cQ} (D(0,n))  =- \frac{1}{2in\pi} \sum_{\pm} \int_{x_n^-}^{x_n^+} \frac{f'( x \pm  i \sqrt{n})}{f( x \pm  i \sqrt{n})} \mathrm{d}x + O\left(\frac{1}{n}\right).
\end{equation}

Now, using (\ref{fPrimesurfPos}), we obtain that
$$\frac{1}{2in\pi} \int_{x_n^-}^{x_n^+} \frac{f'( x +  i \sqrt{n})}{f( x +  i \sqrt{n})} \mathrm{d}x =  O\left(\frac{1}{n}\right),$$
while thanks to (\ref{fPrimesurfNeg}) and (\ref{eq:fPrimeSurfReste}), we get
$$\frac{1}{2in\pi} \int_{x_n^-}^{x_n^+} \frac{f'( x -  i \sqrt{n})}{f( x -  i \sqrt{n})} \mathrm{d}x = \frac{2 \cL_\cQ}{\pi} +  O\left(\frac{1}{n}\right).$$

We thus recover the result of Theorem \ref{Th:DavPush} for unbalanced graphs.

\subsection{Proof of Proposition \ref{Prop:GaussLowerBound}}
Let $x_0\in \R$. We take $y_1 < Y$, $y_2 >0$, and let $a>0$.

We shall  apply (\ref{eq:GaussCount}) with
$$g(z) = e^{-a (z- x_0 - iy_1)^2} = e^{-a (x-x_0)^2 + a (y-y_1)^2 - 2i (x-x_0) (y-y_1)},$$
where $z=x+iy$.

First of all, thanks to (\ref{fPrimesurfNeg}) and (\ref{eq:fPrimeSurfReste}), we have
\begin{align*}
 \left| \int_\R g( x + i y_1) \frac{f'(x+iy_1)}{f(x+iy_1)} \mathrm{d}x  - 2 i\cL_\cQ\sqrt{\frac{\pi}{a}} \right| &\leq 2 \mathcal{L}_\cQ \frac{e^{(y_1-Y) L_{min}}}{1- e^{(y_1-Y) L_{min}}}  \int_\R e^{-a(x-x_0)^2} \mathrm{d}x\\
 &= 2 \mathcal{L}_\cQ \frac{e^{(y_1-Y) L_{min}}}{1- e^{(y_1-Y) L_{min}}} \sqrt{\frac{\pi}{a}}
\end{align*}

In particular, this quantity is smaller than $\frac{\mathcal{L}_\cQ}{4} \sqrt{\frac{\pi}{a}}$ when we take
\begin{equation}\label{eq:Condy1}
y_1 = Y - \frac{\ln 16}{L_{min}}.
\end{equation}

Next, thanks to (\ref{fPrimesurfPos}), we have
\begin{align*} 
 \left| \int_\R g( x + i y_2) \frac{f'(x+iy_2)}{f(x+iy_2)} \mathrm{d}x \right|
 &\leq 2 \cL_\cQ \frac{e^{-y_2 L_{min}}}{1- e^{-y_2 L_{min}}} e^{a (y_2-y_1)^2} \sqrt{\frac{\pi}{a}}.
 \end{align*}
 
In particular, the quantity above is smaller than $\frac{\mathcal{L}_\cQ}{4} \sqrt{\frac{\pi}{a}}$ when we take
\begin{equation}\label{eq:Conda}
\begin{aligned}
y_2 &= \frac{\ln 32}{L_{min}}\\
a& =  \frac{\ln 2}{ \left( 2 \frac{\ln 32}{L_{min}} - Y\right)^2 }.
\end{aligned}
\end{equation}

Therefore, if we take $y_1, y_2$ and $a$ as in (\ref{eq:Condy1}) and (\ref{eq:Conda}), we have
\begin{equation}\label{eq:LowerTrace}
\left|\sum_{j=1,2} \int_\R g( x + i y_j) \mathrm{Tr} \left[ U'_{\cQ}(x + i y_j) \left( \mathrm{Id}- U_{\cQ}(x + iy_j)\right)^{-1} \right] \mathrm{d}x \right| > \frac{\mathcal{L}_\cQ}{2} \sqrt{\frac{\pi}{a}}.
\end{equation}

\paragraph{Upper bounds on resonances in vertical strips}

Now, we turn to the left-hand side of (\ref{eq:GaussCount}). If $x_1< x_2\in \R$, let us denote by $\mathcal{S}_{x_1, x_2}$ the strip $\{z\in \C ; x_1 \leq \Re  z \leq x_2\}$.

Let us denote by $N_0$ a number such that
$$\forall x_1 \in \R,~~ \mathcal{N}_\cQ \left(\cS_{x_1, x_1+\frac{1}{L_{min}}}\right) \leq N_0.$$

The number $N_0$ can be estimated using (\ref{eq:Jensen2}). Indeed, the resonances in $\cS_{x_1, x_1+\frac{1}{L_{min}}}$ do all belong to $D(x_1 +\frac{i}{L_{min}}, |Y|+\frac{1}{L_{min}})$.
Thanks to (\ref{eq:PasDidee}) and (\ref{eq:ControleR}), we have $\left|\ln f\left(x_1+\frac{i}{L_{min}}\right)\right| \geq |B(\cQ)| \frac{e^{-1}}{1- e^{-1}}$, while thanks to (\ref{eq:BorneFGenerale}), we have $\ln \max_{|z-z'|= 2\left(|Y|+ \frac{1}{L_{min}}\right)} |f(z)|\leq 2 \frac{1+ \ln (D+ n_0)}{L_{min}}  |B(\cQ)| L_{max} $. Therefore, (\ref{eq:Jensen2}) implies that we can take

\begin{equation}\label{eq:SuperJensen}
\begin{aligned}
N_0 &\leq \frac{|B(\cQ)|}{\ln 2} \left( 2 L_{max} \frac{1+ \ln (D+ n_0)}{L_{min}}  + \frac{e^{-1}}{1- e^{-1}}  \right)\\
&\leq \frac{\cL_\cQ}{L_{min} \times \ln 2}\left( 2 L_{max} \frac{1+ \ln (D+ n_0)}{L_{min}}  + 0.6 \right).
\end{aligned}
\end{equation}

\paragraph{End of the proof}
Now, we estimate, for any $\alpha >0$
\begin{align*}
\left|\sum_{z\in \mathrm{Res}_\mathcal{Q}\left(\left\{|\Re z - x_0|\geq \frac{\alpha}{L_{min}}\right\}\right)} g(z)\right| &\leq e^{a y^2_1} \sum_{z\in \mathrm{Res}_\mathcal{Q}\left(\left\{|\Re z-x_0|\geq \frac{\alpha}{L_{min}}\right\}\right)} e^{-a |\Re z-x_0|^2}\\
&\leq e^{a y^2_1} \sum_{n\in \N} \sum_{z\in \mathrm{Res}_\mathcal{Q}\left(\left\{\frac{n+1+\alpha}{L_{min}} >|\Re z - x_0|\geq \frac{n+\alpha}{L_{min}}\right\}\right)} e^{-a |\Re z-x_0|^2}\\
&\leq 2 N_0 e^{a y^2_1} \sum_{n\in \N} e^{- \frac{a}{L_{min}^2} (\alpha+n)^2}\\
&\leq 2N_0 e^{a y^2_1} e^{- \frac{a}{L_{min}^2} \alpha^2} \sum_{n \in \N} e^{- a\frac{ n}{L_{min}^2} }\\
&\leq  2N_0  \frac{e^{-a (\frac{\alpha^2}{L_{min}^2}- y^2_1)}}{1-e^{-\frac{a}{L_{min}^2}}} 
\end{align*}

In particular, this quantity is smaller than $\frac{\cL_\cQ}{4}\sqrt{\frac{\pi}{a}} $ provided that

\begin{align*}
e^{-\frac{a}{L_{min}^2}\alpha^2} &\leq \cL_\cQ  e^{-a y^2_1}\frac{1-e^{-\frac{a}{L_{min}^2}}}{8 N_0} \sqrt{\frac{\pi}{a}}\\
&\leq L_{min} \ln 2 e^{-a y^2_1}\frac{1-e^{-\frac{a}{L_{min}^2}}}{8  \left( 2 L_{max} \frac{1+ \ln (D+ n_0)}{L_{min}}  + 0.6  \right)} \sqrt{\frac{\pi}{a}},
\end{align*}
or, in other words,
\begin{equation}\label{eq:CondAlpha}
\frac{\alpha}{L_{min}} \geq \left( y^2_1 - \frac{1}{a} \ln \left( L_{min} \ln 2 \frac{1-e^{-\frac{a}{L^2_{min}}}}{8 \left( 2 L_{max} \frac{1+ \ln (D+ n_0)}{L_{min}}  + 0.6  \right)} \sqrt{\frac{\pi}{a}} \right) \right)^{1/2}.
\end{equation}

Therefore, using (\ref{eq:LowerTrace}), we see that whenever (\ref{eq:CondAlpha}) is satisfied, we have
$$\left|\sum_{z\in \mathrm{Res}_\mathcal{Q}\left(\left\{|\Re z - x_0|\leq \frac{\alpha}{L_{min}}\right\}\right)} g(z)\right| \geq \frac{\cL_\cQ}{4} \sqrt{\frac{\pi}{a}}.$$

Since each term in the sum has a modulus smaller than $e^{a (y_2-y_1)^2} \leq 2$, we deduce that the number of resonances in 
$\left\{|\Re z - x_0|\leq \frac{\alpha}{L_{min}}\right\}$ is at least $\frac{\cL_\cQ}{8} \sqrt{\frac{\pi}{a}}$, as announced.

\section{Benjamini-Schramm convergence for open quantum graphs}\label{sec:BS}

\subsection{Definition of the Benjamini-Schramm convergence}
We now recall the definition of Benjamini-Schramm convergence of open quantum graphs, following closely \cite[\S 3.3.1]{Ing}.

A rooted open quantum graph $(\mathcal{Q},b_0)=(V,E,L,\mathbf{n},b_0)$ will be the data of a quantum graph $\mathcal{Q}=(V,E,L,\mathbf{n})$, and of a bond $b_0\in B(\cQ)$.
\begin{tcolorbox}
\begin{definition}
We say that two rooted quantum graphs $(\mathcal{Q}_0,b_0)= (V_0,E_0,L_0,\mathbf{n}_0,b_0)$ and $(\mathcal{Q}_1,b_1)=(V_1,E_1,L_1,\mathbf{n}_1,b_1)$ are equivalent, which we denote by $(\mathcal{Q}_0,b_0)\sim (\mathcal{Q}_1,b_1)$, if there exists a graph isomorphism $\phi : (V_0,E_0)\longrightarrow (V_1,E_1)$ such that $\phi(o_{b_0})= o_{b_1}$, $\phi(t_{b_0})= t_{b_1}$, $L_1\circ \phi = L_0$, and $\mathbf{n}_1 \circ \phi = \mathbf{n}_0$.

The set of rooted quantum graphs, quotiented by $\sim$, will be denoted by $\mathrm{ROQ}$. If $(\cQ, b_0)$ is a rooted quantum graph, we will denote by $[\cQ,b_0]$ its equivalence class.
\end{definition}
\end{tcolorbox}

If $v\in G$ is a vertex in a graph and $r\in \N$, we write $\mathrm{B}_G(v,r)$ for the set of vertices which are at a (combinatorial) distance at most $r$ from $v$. We write $E(\mathrm{B}_G(v,r))$ for the set of edges in $E$ connecting two vertices of $\mathrm{B}_G(v,r)$.

We introduce a distance between rooted quantum graphs as follows
\begin{align*}
\mathrm{d}\left([\cQ_1,b_1], [\cQ_2,b_2]\right) := \inf \Big{\{} &\varepsilon>0 ~ \big{|} ~  \exists \phi : \mathrm{B}_{G_1}(o_{b_1}, \lfloor\varepsilon^{-1}\rfloor) \to  \mathrm{B}_{G_2}(o_{b_2}, \lfloor\varepsilon^{-1}\rfloor) \text{ a graph isomorphism }\\
& \text{ such that } \mathbf{n}_2 \circ \phi = \mathbf{n}_1 \text{ and } \sup\limits_{e \in E(\mathrm{B}_{G_1}(o_{b_1}, \lfloor\varepsilon^{-1}\rfloor))} |L_2(\phi(b)) - L_1(b)|<\varepsilon \Big{\}}.
\end{align*}

Note that this definition is independent of the representatives we chose in the equivalence classes $[\cQ_1,b_1], [\cQ_2,b_2]$, so it is well-defined on $\mathrm{ROQ}$. Furthermore, one can show that $(\mathrm{ROQ}, d)$ is a Polish space, i.e., a separable complete metric space.

Let $\cP(\mathrm{ROQ})$ be the set of Borel probability measures on $\mathrm{ROQ}$.

\begin{tcolorbox}
\begin{definition}
Any finite quantum graph $\mathcal{Q}=(V,E,L,\mathbf{n})$ defines a probability measure $\nu_{\cQ}\in \cP(\mathrm{ROQ})$ obtained by choosing a root uniformly at random: 
\[
\nu_{\cQ}:= \frac{1}{|B(\cQ)|} \sum_{b_0\in B(\cQ)}  \delta_{[(\cQ,b_0)]}.
\]

If $\textcolor{black}{(\cQ_N)}$ is a sequence of quantum graphs, we say that $\mathbb{P}\in \cP(\mathrm{ROQ})$ is the \emph{local weak limit} of $\textcolor{black}{(\cQ_N)}$, or that $\textcolor{black}{(\cQ_N)}$ \emph{converges in the sense of Benjamini-Schramm to $\mathbb{P}$}, if $\textcolor{black}{(\nu_{\cQ_N})}$ converges weakly to $\mathbb{P}$.
\end{definition}
\end{tcolorbox}

This notion of convergence can be explained as follows. Let $\chi$ be a bounded function on the set of rooted quantum graphs, continuous for the distance $\mathrm{d}$ introduced in the previous paragraph. Then the average value of $\chi\left((\mathcal{Q}_N, b_0)\right)$ when $b_0$ is chosen uniformly at random converges to the expectation $\mathbb{E}_{\mathbb{P}}[\chi]$. 

Let $D, n_0 \in \N$, $0<m\leq M$. We define $\mathrm{ROQ}^{D, n_0, L_{min}, L_{max}}$ and $\mathrm{ROQ}'^{D, n_0, L_{min}, L_{max}}$ as the subsets of $\mathrm{ROQ}$ of equivalence classes $[\cQ,b_0]=[(V,E,L,\mathbf{n}, b_0)]$ 
such that $\cQ\in \mathfrak{Q}_{D, n_0, L_{min}, L_{max}}$ (respectively $\cQ\in \mathfrak{Q}'_{D, n_0, L_{min}, L_{max}}$). The following Lemma can be proven exactly as \cite[Lemma 3.6]{BSQG}. 

\begin{tcolorbox}
\begin{lemme}\label{lem:QCompact}
The subset $\mathrm{ROQ}^{D, n_0,  L_{min}, L_{max}}$ is compact.

 In particular, using Prokhorov's theorem, we see that if $(\mathcal{Q}_N)$ is a sequence of finite open quantum graphs which satisfy \textbf{(Bounds)}, then  there is a subsequence $(\cQ_{N_k})$ which converges in the sense of Benjamini-Schramm (i.e. there exists $\mathbb{P}\in \mathcal{P}(\mathrm{ROQ})$ supported on $\mathrm{ROQ}^{D, n_0, L_{min}, L_{max}}$ such that $\nu_{{Q}_{N_k}}\xrightarrow{w^*} \mathbb{P}$).
\end{lemme}
\end{tcolorbox}



\subsection{Proof of Theorem \ref{Th:BS}}

\begin{tcolorbox}
\begin{lemme}\label{lem:BSCOnv}
Let $D, n_0\in \N$,  let $0<L_{min}\leq L_{max}$, and let $y_1 <Y(D, n_0, L_{min})$, $y_2>0$.

For every $x\in \R$, $j=1,2$, the maps
$$F_{x+i y_j} : \begin{cases}
 \mathfrak{Q}'_{D, n_0, L_{min}, L_{max}} &\longrightarrow \C\\
 [\cQ, b_0] &\mapsto   \left\langle e_{b_0}, U'_{\cQ}(x + i y_j) \left( \mathrm{Id}- U_{\cQ}(x + iy_j)\right)^{-1} e_{b_0} \right\rangle
\end{cases}$$
are continuous and bounded independently of $x$.
\end{lemme}
\end{tcolorbox}

\begin{proof}
The result follows from the expressions
\begin{align*}
U'_{\cQ}(x + i y_2) (\mathrm{Id} - U_\cQ(x+iy_2))^{-1} = \sum_{k=0}^{+\infty} U'_{\cQ}(x + i y_2) (U_\cQ(x+iy_2))^k\\
U'_{\cQ}(x + i y_1) (\mathrm{Id} - U_\cQ(x+iy_1))^{-1} = \sum_{k=1}^{+\infty} U'_{\cQ}(x + i y_1) (U_\cQ(x+iy_1))^{-k},
\end{align*}
from the fact that each $U'_{\cQ}(z) (U_\cQ(z))^{\pm k}$ is continuous (since it depends only on a neighbourhood of size $k$ of $b_0$), and from the exponential decay of the sum.
\end{proof}

\begin{proof}[Proof of Theorem \ref{Th:BS}]
Let $y_1 \leq Y(D, n_0, L_{min}, L_{max})$, $y_2\geq 0$, $\varepsilon>0$, and let $g\in L^1(\Omega_{y_1- \varepsilon,y_2+ \varepsilon})\cap \mathcal{H}(\Omega_{y_1-\varepsilon,y_2+ \varepsilon})$. 
Thanks to Theorem \ref{Prop:GaussCount2}, we have 

\begin{align*}
\langle \mu_{\cQ_N}, \chi \rangle &= -\frac{1}{\cL_{\cQ_N}}  \frac{1}{2i\pi} \sum_{j=1,2} \int_\R g( x + i y_j) \mathrm{Tr} \left[ U'_{\cQ}(x + i y_j) \left( \mathrm{Id}- U_{\cQ}(x + iy_j)\right)^{-1} \right] \mathrm{d}x\\
&= -\frac{1}{\cL_{\cQ_N}} \frac{1}{2i\pi} \sum_{j=1,2} \int_\R g( x + i y_j) \sum_{b_0\in B(\cQ_N)} F_{x+i y_j} ([\cQ_N, b_0])\\
&= - \frac{|B(\cQ_N)|}{\cL_{\cQ_N}} \frac{1}{2i\pi} \sum_{j=1,2} \int_\R g( x + i y_j) \E_{\nu_{\cQ_N}} \left[F_{x+i y_j} \right]. 
\end{align*}
Now, using the dominated convergence theorem, the bound given in Lemma \ref{lem:BSCOnv} and the definition of the Benjamini-Schramm convergence, we see that the integrals above converges to $$\int_\R g( x + i y_j) \E_{\mathbb{P}} \left[F_{x+i y_j} \right].$$

As to the prefactor, we have
\begin{align*}
\frac{|B(\cQ_N)|}{\cL_{\cQ_N}}= \left(\frac{1}{|B(\cQ_N)|} \sum_{b_0\in B(\cQ_N)} L_{b_0}\right)^{-1},
\end{align*}
which converges thanks to the definition of Benjamini-Schramm convergence. The result follows.
\end{proof}

\providecommand{\bysame}{\leavevmode\hbox to3em{\hrulefill}\thinspace}
\providecommand{\MR}{\relax\ifhmode\unskip\space\fi MR }
\providecommand{\MRhref}[2]{%
  \href{http://www.ams.org/mathscinet-getitem?mr=#1}{#2}
}
\providecommand{\href}[2]{#2}

\end{document}